%% file: toricDuality.tex
\documentclass[11pt]{amsart}

\def\semicolon{\nobreak\mskip2mu\mathpunct{}\nonscript\mkern-\thinmuskip{;}
\mskip6muplus1mu\relax} 

\input{somedef}

\usepackage{etex}
\usepackage[utf8]{inputenc}
\usepackage{graphicx}
\usepackage{epstopdf}
\usepackage[all]{xy}
\usepackage{tikz}  
  \usetikzlibrary{arrows,positioning}
  \tikzset{%
    baseline=-2.3pt,
    text height=1.5ex, text depth=0.25ex,
    >=stealth,
    node distance=2cm,
    mid/.style={fill=white,inner sep=2.5pt},
  }

\usepackage{amsmath}
\newtheorem{theo}{Theorem}[section]
\newtheorem{prop}[theo]{Proposition}

\newtheorem{coro}[theo]{Corollary}
\newtheorem{lemm}[theo]{Lemma}

\theoremstyle{definition}
\newtheorem{defi}[theo]{Definition}

\theoremstyle{remark}
\newtheorem{rema}[theo]{Remark}

\DeclareMathOperator{\dv}{\mathbf{div}}
\DeclareMathOperator{\mon}{\mathbf{mon}}
\DeclareMathOperator{\Tot}{Tot}

\begin{document}

\thispagestyle{plain}

\title{Self-Duality for  Landau--Ginzburg models}
\author{B.~Callander, E.~Gasparim, R.~Jenkins, L.~M.~Silva}


\maketitle


\begin{abstract}
  P.~Clarke describes mirror symmetry as a duality between Landau--Ginzburg models, so that the dual of an LG model is another LG model. 
  We describe examples in which the underlying space is a total space of a vector bundle on the projective line, and we show that self-duality occurs in precisely two cases: the cotangent bundle and the resolved conifold.
\end{abstract}

\label{first}

\section{Introduction}
For us a Landau--Ginzburg  model (LG)  is  a variety
 $X$ together with a regular function $W\colon X \rightarrow \mathbb C$ called the superpotential.
Clarke \cite{clarke} showed that one can state a generalised version of the Homological Mirror Symmetry conjecture of Kontsevich \cite{Ko}
as a duality between LG models. He also showed that this 
correspondence generalises those of Batyrev--Borisov, 
Berglung--H\"ubsch, Givental, and Hori--Vafa. 

This paper is an exercise in understanding the details of this correspondence. We summarise the construction in \cite{clarke}, 
which, for a given LG model $(X,W)$, produces a dual $(X^\vee, W^\vee)$. When $(X^\vee, W^\vee)\cong(X,W)$, we call 
$X$ self-dual. We
then study the case when $X$ is the total space of a vector bundle on $\mathbb P^1$ and 
prove that self-duality occurs in only two cases: $X=\Tot(\mathcal O(-2))$ and $X=\Tot(\mathcal O(-1)\oplus \mathcal O(-1))$.

\section{The Character to Divisor Map}

Let $X$ be a toric variety of rank $n$ with a torus embedding $\iota\colon T\longrightarrow X$. The torus $T=(\mathbb C^*)^n$ is an algebraic group, whose algebraic functions are characters, that is, group morphisms, $\chi\colon T\longrightarrow \mathbb C^*$. Let $M$ denote the group of characters of $T$, and $N$ the group of one-parameter subgroups, naturally identified with the dual of $M$, $\mathrm{Hom}_\mathbb Z(M,\mathbb Z)$. Let $M_\mathbb R$ and $N_\mathbb R$ denote the tensor products $M\otimes_\mathbb Z \mathbb R$ and $N \otimes_\mathbb R \mathbb Z$, respectively.

Since $\iota(T)$ is dense inside $X$, each character $\chi\in M$ can be thought of as a rational map, $f_\chi \colon X \dashrightarrow \mathbb C$, which is nowhere zero on $\iota(T)$. 
Let $R=\{D_1,\ldots,D_r\}$ denote the set of irreducible components of $X\setminus \iota(T)$. These are prime $T$-invariant Weil divisors and can be read off the moment polytope for $X$. Since each $D\in R$ is irreducible and $X$ is normal,
one can compute the order of vanishing, $\textrm{ord}_D(f_\chi)$, of $f_\chi$ along $D$. This defines a map,
\[
\dv(X)\colon M \longrightarrow \mathbb Z^R;\qquad \chi \mapsto \left( \textrm{ord}_{D_1}(f_\chi), \ldots, \textrm{ord}_{D_r}(f_\chi)\right).
\] Choosing ordered generators for $M$ and an ordering of $R$ gives a matrix $M_{\dv}(X) \in \mathrm{Mat}_{n\times r}(\mathbb Z)$. For each $D_k\in R$, let $v_k\in N$ be a generator for the corresponding ray in the fan. By \cite[Section 3.3]{Ful93}, $\textrm{ord}_{D_k}(f_\chi)=\langle \chi, v_k\rangle$. This implies that, when the bases of $N$ and $M$ are dual, the rows of the matrix $M_{\dv} (X)$ are simply the generating vectors, $v_k$.

The cokernel of $\dv(X)$ is the \textbf{Chow group} of $X$, written $A_{n-1}(X)$. When $X$ is a complete toric variety, the Chow group can be identified with the second integral cohomology $H^2(X,\mathbb Z)$ and is torsion free. The following lemma is from \cite{clarke}.

\begin{lemm}
  \label{thm:splitBundle}\cite[Cor. 4.5]{clarke}
  If $D_1, \dotsc, D_c$ are $T$-invariant Cartier divisors and $X$ is the  total space
  of the split  bundle $\mathcal O_Y (-D_1) \oplus \dotsb \oplus \mathcal
  O_Y (-D_c)$ over a toric variety $Y$, then the character group of $X$ decomposes as
  \[
    M_X \cong M_Y \oplus \mathbb Z\sigma_1 \oplus \dotsb \oplus \mathbb Z\sigma_c,
  \]
  where $\sigma_j$ is a rational section of $\mathcal O_Y (D_j)$ whose divisor
  is $D_j$, interpreted here as a character of $T$.  The $T$-invariant Weil divisors of $X$ are the preimages under
  $p$ of the $T$-invariant Weil divisors of $Y$ as well as the total spaces
  $X_j$ of the $c$ subbundles $E_j^\vee$, where $E_j^\vee$ is the dual bundle
  to $\ker (\pi_j \colon E \to \mathcal O (D_j))$.  Furthermore,
  \[
    \dv_X
    =
    \begin{pmatrix}
     \quad \dv_Y & \vert\quad D_1 \quad\vert\quad \dotsb \quad\vert\quad D_c \quad \phantom{.}\\
     \quad 0 & \id
    \end{pmatrix}.
  \]
  with respect to the decomposition of $M_X$ above and 
  \[
    \mathbb Z^{R_X} 
    = 
    \mathbb Z^{R_Y} \oplus \mathbb Z X_1 \oplus \dotsb \oplus \mathbb Z X_c.
  \]
\end{lemm}

\section{The Infinitesimal Action on Monomials}

Let $E$ be a vector bundle on a K\"ahler manifold $Y$ with a global section $w\in H^0(Y,E)$. Assume that $X=\mathrm{Tot}(E^\vee)$ is a toric variety.
A \textbf{superpotential} $W \colon X \rightarrow \mathbb C $  is a regular function on $X$. It can be determined by $w$ as follows. In the category of coherent $\mathcal{O}_Y$-modules, there are isomorphisms
\[
H^0(Y,E)\cong \mathrm{Hom}(\mathcal{O}_Y,E) \cong \mathrm{Hom}(E^\vee,\mathcal{O}_Y).
\] Thus, $w$ determines a morphism from $E^\vee$ to $\mathcal{O}_Y$,
  or, equivalently, a regular function $W$ on the total space of $E^\vee$. Since $T$ acts freely on the embedded torus $\iota(T)\subset X$, the zeroes of the function $W$ must lie on the locus of $T$-invariant divisors. Thus, $W\circ\iota\colon T\rightarrow \mathbb C^*$ is a homomorphism of algebraic groups, which may be expressed as a finite linear sum of characters of $T$:
\[
\iota^*W=\sum_{i=1}^s a_i \xi_i,
\] for scalars $a_i\in \mathbb C$ and characters $\xi_i\in M$. Set $\Xi :=  \{\xi_1,\ldots,\xi_s\}$.

The scalars $\{a_1,\ldots,a_s\}$  depend on the initial choice of embedding $\iota$. In turn, the map $\iota$ is determined by a point $x\in X$; namely, the image of $1\in T$. Write $\iota_x$ for 
the map sending $1$ to $x$. If $x'=tx$ is another point in $\iota(T)$ for some $t\in T$, we have
\[
\iota_{x'}^*W=\sum_{i=1}^s a_i \xi_i(t) \xi_i.
\] 
Let $(\mathbb C^*)^\Xi$ denote the space of all $\mathbb C^*$-linear sums of monomials in $\Xi$ -- these are regular functions on $T$. Now $T$ acts on $(\mathbb C^*)^\Xi$ as above; that is, if $\iota_x^*W\in (\mathbb C^*)^\Xi$ and $t\in T$, then $t\cdot \iota_{x}^*W :=  \iota_{t\cdot x}^*W$. In order to eliminate the dependence of $\iota^*W$ on the choice of embedding, we consider $\iota^*W$ as an element of the quotient $(\mathbb C^*)^\Xi/T$.  The kernel of the exponential map $\mathbb C^n\longrightarrow T; \quad (t_1,\ldots,t_n)\mapsto (\e^{t_1},\ldots,\e^{t_n})$ is isomorphic to $\mathbb Z^n$, as is the lattice of one-parameter subgroups $N$.
Let $\mathbb Z^\Xi$ denote the kernel of the corresponding exponential map on $\mathbb C^\Xi$. The action of $T$ on $(\mathbb C^*)^\Xi$ gives a map $f\colon T\longrightarrow (\mathbb C^*)^\Xi; \quad t\mapsto t\cdot (\xi_1+\cdots+\xi_s)$. Restricting the  derivative  $\d f\colon \mathbb C^r\longrightarrow \mathbb C^\Xi$  to the kernel $N$ of $\e^{(-)}$ yields a map which we denote by  
$$\mon\colon N \longrightarrow \mathbb Z^\Xi\text{.}$$  Hence, the maps $f$, $\d f$, and $\mon$ define a morphism of the following short exact sequences.
\[
\xymatrix{
0 \ar[r] & N      \ar[r] \ar[d]^\mon  &  \mathbb C^n \ar[r]^{\e^{(-)}} \ar[d]^{\d f} & T\ar[d]^f \ar[r] &0 \\
0 \ar[r] & \mathbb Z^\Xi \ar[r]							&  \mathbb C^\Xi \ar[r]^{\e^{(-)}} & (\mathbb C^*)^\Xi  \ar[r] &0
} 
\] 
Choosing  an  ordered basis for $N$ and  an  ordering of the monomials in $\Xi$  allows us  to express the map $\mon$ as a matrix   $M_{\mon}(X)\in \mathrm{Mat}_{n\times s}(\mathbb Z)$ such that  the $k^\mathrm{th}$ row of this matrix is given by the n-tuple $(b_1,\ldots,b_n)$ defined by the equation $\xi_k(t_1,\ldots,t_n)=t_1^{b_1}\cdots t_n^{b_n}$.

\section{Toric LG Models}

A \textbf{toric  Landau--Ginzburg model} is a triple, $(X,W,K)$, where $X$ is a
toric variety, $W$ is a regular function on $X$ and $K\in
A_{n-1}(X)\otimes_{\mathbb Z}\mathbb C/\mathbb Z$ is an element of the Chow group (with $\mathbb C
/ \mathbb Z$ coefficients).  To such a  model we have associated linear maps
$\dv(X)$ and $\mon(X)$. Choosing an element $L\in
\mathrm{coker}(\mon)\otimes_{\mathbb Z}\mathbb C/\mathbb Z$ determines the \textbf{linear data} associated
to $(X,W,K)$; namely, the pairs $(\dv, K)$ and $(\mon, L)$. We now provide an
inverse to this construction.

First we specify the conditions on $\mathbb R$-linear data $(C,c)$ for it to yield
an appropriate toric variety.  Let $C \colon M \to \mathbb Z^r$ be a linear map, and $c \in \mathbb Z^r$. 
 We
say that the $\mathbb R$-linear data $(C,c)$ is \textbf{kopaseptic} if
\begin{enumerate}
  \item \label{itm:polyNonEmpty} the polyhedral set $P = \{ \xi \in M \semicolon C\xi + c \ge 0 \}$ associated to $(C,c)$
    has non-empty interior; and
  \item \label{itm:stdGens} there exists a surjection $k \colon \mathbb Z^r \to \mathbb
    Z^{R_{X (C,c)}}$ sending standard generators either to standard generators
    or to zero such that the following diagram commutes
    \[
      \begin{tikzpicture}[node distance=1cm]
        \node (M) {$M$};
        \node (Zr) [right=of M] {$\mathbb Z^r$};
        \node (ZR) [below=of Zr] {$\mathbb Z^{R_{X(C,c)}}$};
        \draw[->] (M) to node [auto] {$C$} (Zr);
        \draw[->] (M) to node [auto,swap] {$\dv_{X(C,c)}$} (ZR);
        \draw[->] (Zr) to node [auto] {$k$} (ZR);
      \end{tikzpicture}
   \text{,} \]
   where  $R_{X (C,c)}$  denotes the number of torus-invariant divisors of the toric variety $X (C,c)$.
   
\end{enumerate}
Condition \ref{itm:polyNonEmpty} guarantees that the toric variety
$X (C,c)$ corresponding to the polyhedral set of $(C,c)$ is well-defined, and thus allows us to make
sense of condition \ref{itm:stdGens}.  Some of the inequalities $C\xi + c \ge
0$ defining the polyhedral set may be redundant and condition \ref{itm:stdGens} tells us how to remove
these redundances.  In fact, $k$ is almost uniquely determined, the only
choice being which redundant condition to drop.

Now we need to determine precisely when a potential $W$ (defined on a toric
variety $X$) is regular. 
Since it is regular if and only if all its monomials  are regular, and the
$\mon$ matrix encodes all the information about those monomials, we can state
our condition in terms of that matrix.  Indeed, a monomial $\xi$ is regular if and
only if $\dv \xi \ge 0$, which implies the following lemma.

\begin{lemm}  $W$ is regular if and only if
$\dv \circ \mon^T \ge 0$.
\end{lemm}

We now combine the above remarks into one definition. Let $A$ and $B$ be homomorphisms of free  abelian groups of finite rank such
that the domains of $A$ and $B$ have the same rank, and let $K$ and $L$ be
elements in $\mathrm{coker}(A)\otimes_\mathbb Z \mathbb C/\mathbb Z$ and
$\mathrm{coker}(B)\otimes_\mathbb Z \mathbb C/\mathbb Z$, respectively. A pair $(A,K)$ and
$(B,L)$ is called \textbf{$\mathbb C/\mathbb Z$-linear data}. Such data is said to be
\textbf{kopaseptic} if
\begin{enumerate}
  \item $(A, (\Im{K}))$ is kopaseptic; and
  \item the entries of the matrix $A \circ B^T$ are all non-negative.
\end{enumerate}
Here $\Im{K}$  denotes  the imaginary part of $K$.

Given kopaseptic $\mathbb C/\mathbb Z$-linear data $(A,K)$, $(B,L)$, we can define the
corresponding toric  Landau--Ginzburg model $(X,W,K)$ given by
\begin{enumerate}
  \item the toric variety $X := X(A,\Im{K})$ determined by $A$ and
      $\Im{K}$;
  \item the regular function $W := W (B,L)$ determined by $B$ and $L$.
\end{enumerate}
The element $K$ specifies a choice of complexified Kähler class for our
 Landau--Ginzburg model.

\section{Self-duality}

Let $(X,W,K)$ be a toric  Landau--Ginzburg model with linear data $(\dv(X),K)$, $(\mon,L)$. Then the \textbf{  dual}  $(X^\vee, W^\vee, K^\vee)$ of $(X,W,K)$ is the toric  Landau--Ginzburg model corresponding to the linear data obtained exchanging $(\dv, K)$ and $(\mon, L)$.

\begin{lemm}
Let $(X,W,K)$ and $(Y,W',K')$ be toric  Landau--Ginzburg models. Then $(X\times Y,W+W',K+K')$ is a toric  Landau--Ginzburg model and $\dv(X\times Y)=\dv(X)\oplus \dv(Y)$ and $\mon(X\times Y)=\mon(X)\oplus\mon(Y)$.
\end{lemm} 
\begin{proof} This follows directly from the definitions, given that the torus action on $X\times Y$ agrees with the original actions on $X$ and $Y$.\end{proof}
This immediately implies the following.
\begin{coro} Suppose $(X,W,K)$ is a toric  Landau--Ginzburg model which is dual to $(X^\vee,W',K')$. Then $(X\times X^\vee, W+W', K+K')$ is self-dual.
\end{coro}

\subsection{The CY Condition}

There are several inequivalent definitions of a Calabi--Yau manifold. Some authors require that the manifold be a compact complex K\"ahler manifold with a Ricci flat metric, while others use a stronger condition that implies the former: a compact complex K\"ahler manifold with trivial canonical bundle. When a K\"ahler manifold is non-compact, the triviality of the canonical bundle does not necessarily  imply the  existence of a complete Ricci flat  metric. 
In this case we make the following definition.

\begin{defi} A complex K\"ahler manifold is \textbf{Calabi--Yau} if it has trivial canonical bundle and admits a complete Ricci-flat metric. Such a metric is called a \textbf{Calabi--Yau metric}.
\end{defi}

The dual of a Calabi--Yau variety is expected to also be Calabi--Yau. 
\section{Self-duality for Bundles on $\mathbb P^1$ }

We now describe  such dualities for the case when our  variety $X$ is the total space of a vector bundle on $\mathbb P^1$.

\subsection{Rank 1}

Let 
$X=\Tot(\mathcal{O}_{\mathbb P^1}(-k))$. For $k<0$, $E$ has no global sections, so assume $k\ge 0$. The chart $U :=  \{[z:1] \semicolon z\in \mathbb C\}$ of $\mathbb P^1$ determines a chart of $X$ on which points may be described as pairs $(z,u)$, where $u$ is the coordinate for the fibre of $E^\vee|_U$. The point $x=(1,1)$ determines the embedding $\iota_x$, so that an element $(t_1,t_2)\in T$ acts on $X$ by $(t_1,t_2)\cdot (z,u)=(t_1z,t_2u)$. Having embedded the torus this way, Laurent polynomials in $t_1$ and $t_2$ can be interpreted both  as
characters of the torus $T$ and 
as rational  functions on $X$. This gives a basis for the group of characters $M=\langle t_1,t_2\rangle$. Let $\nu_1,\nu_2$ be the dual basis for the one-parameter subgroups $N$. The $T$-invariant divisors of $X$ are $f_0=\{t_1=0\}$, $f_\infty = \{t_1=\infty\}$ and $\ell=\{t_2=0\}$. The moment polytope for $X$ is given by connecting the vertices $(0,1)$-$(0,0)$-$(1,0)$-$(k+1,1)$.
Fig.~\ref{fig:Z2polytope} illustrates the case  $k=2$.

\begin{figure}[ht]
  \centering
  \includegraphics[scale=1.3]{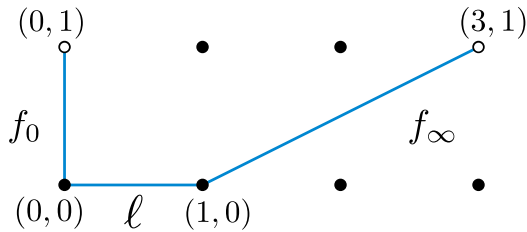}
  \caption{The moment polytope of $\Tot(\mathcal{O}(-2))$ with invariant divisors $\ell$, $f_0$, and $f_\infty$}
  \label{fig:Z2polytope}
\end{figure}

\begin{rema} The unique value of $k$ for which $X$ is Calabi--Yau is $k=2$. \end{rema}

\begin{prop}\label{prop:Z_k} The toric variety $X=\Tot(\mathcal{O}_{\mathbb P^1}(-k))$ belongs to a   self-dual  Landau--Ginzburg model $(X,W,K)$ if and only if $k=2$. 
\end{prop}

\begin{proof}
With respect to the fixed basis above, the rows of the $\dv$-matrix are given by the  vectors normal to the
edges of the moment polytope,  which are $(1,0)$, $(0,1)$, and $(-1,k)$. Hence
\[
  M_{\dv}(X)
  =
  \begin{pmatrix}
    1 & 0 \\
    -1 & k \\
    0 & 1 
  \end{pmatrix}.
\]
A global section $w$ of $E$ is represented by a polynomial of degree $k$ (we assume  $k\ge 0$). Identifying 
$\mathbb P^1$
with the subvariety of $X$ cut out by
 $t_2=0$ gives a superpotential $W= a_0 t_2 + a_1 t_1t_2 + \cdots + a_k t_1^kt_2$ for some $a_0,\ldots,a_k\in \mathbb C$. 
 For $X$ to belong to a self-dual toric  Landau--Ginzburg model,
there must exist a choice of basis for $N$ and an ordering of $\Xi$ such that $M_{\dv}(X)=M_{\mon}(X)$. Clearly, $\Xi$ must have cardinality three, so $\Xi$ is a subset of three of the monomials in $\{t_2, \ldots, t_1^kt_2\}$. With the dual basis for $M$, the $\mon$-matrix for $X$ is given by
\[
  M_{\mon}(X)=
  \begin{pmatrix}
    a   & 1\\
    b & 1\\
    c & 1
  \end{pmatrix},
\] where $a,b,c$ are distinct integers in $\{0,\ldots,k\}$. If a choice of basis for $N$ exists such that $M_{\mon}(X)=M_{\dv}(X)$, then there are (non-zero) integers $\lambda,\mu\in \mathbb Z$ such that $\lambda(a,b,c) + \mu(1,1,1)= (1,-1,0)$. This implies $a+b-2c=0$. Likewise, there exist (non-zero) integers $\lambda',\mu'\in \mathbb Z$ such that $\lambda'(a,b,c)+\mu(1,1,1)= (0,k,1)$. This implies $(k-1)a+b-kc=0$. Together these two equations give $(k-2)(a-c)=0$, which, since $a$ and $c$ are distinct, implies that $k=2$.

It remains to show that, for $k=2$, an element $K\in A_{n-1}\otimes_\mathbb Z \mathbb C/\mathbb Z$ can be chosen so that $(\dv,\mathfrak {Im}(K))$ is kopaseptic. The Chow group in this case is isomorphic to $\mathbb Z$ by an isomorphism sending the generator $(1,1,-2)$ in the codomain of $M_{\dv}(X)$ to $1\in \mathbb Z$. The polyhedral set defined by choosing $t>0\in A_{n-1}$ has non-empty interior and produces inward normals that determine the fan for $X$.  On the other hand, for $t\le 0$, the relation from the third row of $M_{\dv}(X)$ is made redundant. It follows that lifting $(1,1,-2)$ to $\mathbb C/\mathbb Z$ gives a $K$ such that $(X,W,K)$ is   self-dual.
\end{proof}

\subsection{Rank Two Bundles}

Now we consider the rank 2 bundles on $\mathbb P^1$ whose total space is Calabi--Yau, so  $E=\mathcal O(-k) \oplus \mathcal O(k+2)$ on $Y=\mathbb P^1$. Let $X=W_k := \Tot \left(E^\vee\right)$. Note that $W_k \simeq W_{-k-2}$, so we can assume  $k\ge -1$.
 
\begin{prop}
  \label{prop:W_k} 
  The toric variety $X=W_k$ belongs to a   self-dual toric  Landau--Ginzburg model $(X,W,K)$ if and only if $k=0,-1$.
\end{prop}

\begin{proof}
As in the example above, the chart $U :=  \{[z:1] \semicolon z\in \mathbb C\}$ of $\mathbb P^1$ gives a chart on $X$ on which points may be described as triples, $(z,u,v)$, where $u$ is the coordinate along a fibre of $\mathcal{O}(k)|_U$ and $v$ is the coordinate along a fibre of $\mathcal{O}(-k-2)|_U$. Let $T=(\mathbb C^*)^3$ be embedded in $X$ so that $(t_1,t_2,t_3)\in T$ acts  by the rule $(t_1,t_2,t_3)\cdot (z,u,v)=(t_1z,t_2u,t_3v)$.
Again we let Laurent polynomials in $t_i$ represent both the characters of $T$ and the rational functions on $X$. With this notation, the $T$-invariant divisors are $f_0=\{t_1=0\}$,
$f_\infty = \{ t_1=\infty \}$, $l_1=\{t_2=0\}$ and $l_2=\{t_3=0\}$. Let $[\infty]$ denote the divisor of $\mathbb P^1$ which is the intersection of $\mathbb P^1$ with $f_\infty$ in $X$. Applying Lemma~\ref{thm:splitBundle} with $c=2$, $D_1 = -k[\infty]$, $D_2 = (k+2) [\infty]$, $\sigma_1= t_2$ and $\sigma_2=t_3$ gives the matrix
\[
  M_{\dv}(X)
  =
  \begin{pmatrix}
    1 & 0 & 0 \\
    -1 & -k & k+2 \\
    0 & 1 & 0 \\
    0 & 0 & 1
  \end{pmatrix}.
\]
The following three cases describe the global sections of $E= \mathcal{O}(-k)\oplus \mathcal{O}(2+k)$ on $\mathbb P^1$.
\[
H^0(\mathbb P^1, E) = \begin{cases}  H^0(\mathbb P^1, \mathcal{O}(1)\oplus \mathcal{O}(1) ) \cong \mathbb C[x]_1\oplus \mathbb C[x]_1 & \textrm{ when $k=-1$,} \\
															H^0(\mathbb P^1, \mathcal{O}(2)\oplus \mathcal{O} ) \cong \mathbb C[x]_2 \oplus \mathbb C  & \textrm{ when $k=0$,}\\
															H^0(\mathbb P^1, \mathcal{O}(k+2) ) \cong \mathbb C[x]_{k+2} & \textrm{ when $k\ge 1$.} \end{cases}
\] 

When $k\ge 0$ the $\dv$ and $\mon$ matrices decompose into the direct sum of the $\dv$ and $\mon$ matrices for $\Tot(\mathcal{O}(-k-2))$ with the identity matrix. 
That $X$ belongs to a   self-dual  Landau--Ginzburg model for $k=0$ but not for $k\ge 1$ follows from Proposition \ref{prop:Z_k}.

Consider $k=-1$. A generic section of $E$ is a pair of linear polynomials in a single variable. This produces the superpotential $W= a_0t_2 + a_1t_1t_2 + b_0t_3+b_1t_1t_3$ on $X$, where $a_0,a_1,b_0,b_1\in \mathbb C$. Judiciously order the monomials in $W$ so that $\Xi = \{t_1t_3,t_2,t_3,t_1t_2\}$. Let $s_1,s_2,s_3$ denote one-parameter subgroups dual to the characters $t_1,t_2,t_3$. Finally, choose the basis $N=\langle s_1s_3, s_3, s_1s_2\rangle$. With respect to these choices, $M_{\mon}(W_k)=M_{\dv}(W_k)$. 
The Chow group is isomorphic to $\mathbb Z$, which we identify with the subgroup $\{(t,t,-t,-t) \semicolon t\in \mathbb Z\}$ of the codomain of $M_{\dv}(X)$. Again, if $t<0$, then the relations from the third and forth rows of $M_{\dv}(X)$ are redundant, but $t>0$ produces a polytope with inward normals which define the fan for $X$. Choosing a lifting of $(1,1,-1,-1)$ yields the required Chow group element $K$.
\end{proof}

\subsection{Higher Rank Bundles}

We recall the following definition.

\begin{defi}
A vector bundle on a curve is \textbf{polystable} if it is isomorphic to a sum of stable bundles with the same slope.
\end{defi}

The following theorem of Hori is from \cite[Theorem 32.8.8]{Hori}.
\begin{theo}(Hori) A holomorphic vector bundle admits a Calabi--Yau metric if and only if it is polystable.
\end{theo}

\begin{theo} 
  \label{thm:sdP1}
  Let $X$ be the total space of a vector bundle on $\mathbb P^1$. Suppose, additionally, that such a bundle is Calabi--Yau. Then $X$ is self-dual if and only if   $X = \mathcal{O}(-2)$ or  $X =\mathcal{O}(-1)\oplus\mathcal{O}(-1)$.
\end{theo}

\begin{proof}
  The previous sections deal with the rank one and two cases. 
  The Grothendieck splitting lemma states that a rank $n$ bundle $E$ on $\mathbb P^1$ splits as a sum of line bundles $E\cong \mathcal{O}(a_1)\oplus\cdots\oplus\mathcal{O}(a_n)$. 
  The total space of $E$ has trivial canonical bundle if and only if $\sum a_i = -2$. 

  If $E$ is a sum of two line bundles, $\mathcal{O}(a)\oplus\mathcal{O}(b)$, with $a\ge b$, then the slope of $\mathcal{O}(a)$ is greater than or equal to the slope of $E$.  
  Induction on the rank $r$ of $E$, for $r \ge 2$, shows that vector bundles on $\mathbb P^1$ of rank $r \ge 2$ are not stable.  
  Thus, the only stable vector bundles on $\mathbb P^1$ are the line bundles. 
  It follows that a vector bundle on $\mathbb P^1$ is polystable if and only if it is of the form
  \[
    \mathcal{O}(a)\oplus \cdots \oplus \mathcal{O}(a)
  \] 
  for some $a$.
  Therefore, vector bundles on $\mathbb P^1$ with rank greater than two do not satisfy the Calabi--Yau condition required for self-duality. 
\end{proof}

\begin{rema}
  We expect that the hypothesis that the bundle is Calabi--Yau can be removed from this theorem.
\end{rema}

\begin{rema}
  The Calabi--Yau condition used in Theorem~\ref{thm:sdP1} is stronger than the commonly  used definition that only requires  triviality of the  canonical bundle.
  Using the latter, one can apply the algebraic argument from the proof of Proposition~\ref{prop:Z_k} to show that a Calabi--Yau vector bundle on $\mathbb P^1$ can also be a direct sum of $\mathcal O (-2)$ or $\mathcal O (-1)^{\oplus 2}$ with $\mathcal O^{\oplus k}$ for some $k \ge 0$.
  That the former, stronger condition removes the trivial summands gives a justification of its suitability. 
\end{rema}

\section*{Acknowledgements}

  We thank Patrick Clarke for patiently explaining details of his work, and Tony Pantev for enlightening discussions. 

  Elizabeth Gasparim was supported by Fapesp grant 2012/10179-5 and Rollo Jenkins was supported by Faepex-PRP-Unicamp under grant number PRP/FAEPEX 005/2014 and Fapesp grant 2013/17654-3.

Brian Callander \\
{\it E-mail address}: {\tt briancallander@gmail.com }\\[0.3cm]
Rollo Jenkins \\
{\it E-mail address}: {\tt rollojenkins@gmail.com }\\[0.3cm]
Elizabeth Gasparim \\
{\it E-mail address}: {\tt etgasparim@gmail.com }\\[0.3cm]
Lino  Marcos Silva \\
{\it E-mail address}: {\tt linofriend@gmail.com }\\[0.3cm]
Address of all authors:\\
IMECC \\
UNICAMP \\
Universidade Estadual de Campinas \\
Rua S\'{e}rgio Buarque de Holanda 651 \\ 
Cidade Universit\'{a}ria ``Zeferino Vaz'' \\
Bar\~{a}o Geraldo \\ 
Campinas \\
S\~{a}o Paulo \\
Brazil

\label{last}
\end{document}

%% file: somedef.tex
\newcommand{\e}{\mathrm{e}}

\renewcommand{\d}{\mathrm{d}}
\newcommand{\id}{id}

\theoremstyle{plain}